\documentclass[12pt]{amsproc}
\usepackage{enumerate} 
\makeatletter
\@namedef{subjclassname@2020}{\textup{2020} Mathematics Subject Classification}
\makeatother
\newtheorem{theorem}{\bf Theorem}[section]
\newtheorem{lemma}[theorem]{\bf Lemma}
\newtheorem{proposition}[theorem]{\bf Proposition}
\newtheorem{corollary}[theorem]{\bf Corollary}
\newtheorem{remark}[theorem]{\sc Remark}

\newtheorem{example}[theorem]{\sc Example}

\newcommand{\pr }{\mathrm{Pr} }
\newcommand{\prp}{\mathrm{pr} }
\newcommand{\ep}{\epsilon}

\DeclareMathOperator{\frat}{Frat}

\DeclareMathOperator{\PSL}{PSL}

\DeclareMathOperator{\A}{A}

\DeclareMathOperator{\Sym}{S}
\DeclareMathOperator{\PSp}{PSp}

\begin{document}
\title[Commuting probability]{Commuting probability for the Sylow subgroups of a finite group}
\thanks{The first three authors are members of GNSAGA (INDAM), 
and the fourth author was  supported by  FAPDF and CNPq.
}

\author[E. Detomi]{Eloisa Detomi}
\address{Dipartimento di Matematica \lq\lq Tullio Levi-Civita\rq\rq, Universit\`a degli Studi di Padova, Via Trieste 63, 35121 Padova, Italy} 
\email{eloisa.detomi@unipd.it}
\author[A. Lucchini]{Andrea Lucchini}
\address{Dipartimento di Matematica \lq\lq Tullio Levi-Civita\rq\rq, Universit\`a  degli Studi di Padova, Via Trieste 63, 35121 Padova, Italy} 
\email{lucchini@math.unipd.it}
\author[M. Morigi]{Marta Morigi}
\address{Dipartimento di Matematica, Universit\`a di Bologna\\
Piazza di Porta San Donato 5 \\ 40126 Bologna \\ Italy}
\email{marta.morigi@unibo.it}
\author[P. Shumyatsky]{Pavel Shumyatsky}
\address{Department of Mathematics, University of Brasilia\\
Brasilia-DF \\ 70910-900 Brazil}
\email{pavel@unb.br}


\begin{abstract} 
 For subsets $X,Y$ of a finite group $G$, let $\pr(X,Y)$ denote the probability that two random elements $x\in X$ and $y\in Y$ commute.  Obviously, a finite group $G$ is nilpotent if and only if $\pr(P,Q)=1$ whenever $P$ and $Q$ are Sylow subgroups of $G$ of coprime orders. 
 
Suppose that $G$ is a finite group in which for any distinct primes $p,q\in\pi(G)$ there is a Sylow $p$-subgroup $P$ and a Sylow $q$-subgroup $Q$ of $G$ such that $\pr(P,Q) \ge \epsilon$. We show that $F_2(G)$ has $\epsilon$-bounded index in $G$.

If $G$ is a finite soluble group in which for any prime $p\in\pi(G)$ there is a Sylow $p$-subgroup $P$ and a Hall $p'$-subgroup $H$ such that $\pr(P,H)\ge \epsilon$, then $F(G)$ has $\epsilon$-bounded index in $G$.

Moreover, we establish  criteria for nilpotency and solubility of $G$ such as: If for any primes $p,q\in\pi(G)$ the group $G$ has a Sylow $p$-subgroup $P$ and a Sylow $q$-subgroup $Q$ with $\pr(P,Q)>2/3$, then $G$ is nilpotent. If for any primes $p,q\in\pi(G)$ the group $G$ has a Sylow $p$-subgroup $P$ and a Sylow $q$-subgroup $Q$ with $\pr(P,Q)>2/5$, then $G$ is soluble.

  \end{abstract}
 \maketitle

\section{Introduction} 
Given two subsets $X,Y$ of a finite group $G$, we write $\pr(X,Y)$ for the
  probability that random elements $x\in X$ and $y\in Y$ commute. 
The number $\pr(G,G)$ is called the commuting probability of $G$.  It is well-known that $\pr(G,G)\leq5/8$ for any nonabelian group $G$. Another important result is the theorem of P. M. Neumann \cite{neumann} which states that 
 if $G$ is a finite group and $\epsilon$ is  a positive number such that $\pr(G,G)\geq\epsilon$,  
 then $G$ has a normal subgroup $R$ such that both the index $|G:R|$ 
and the order of the commutator subgroup $[R,R]$ are $\epsilon$-bounded  (see also \cite{eberhard}).

Throughout the article we use the expression ``$(a,b,\dots)$-bounded" to mean that a quantity is bounded from above by a number depending only on the parameters $a,b,\dots$.

A number of further results on commuting probability in finite groups can be found in \cite{gr,bgmn}.
There are several recent papers studying $\pr(H,G)$, where $H$ is a subgroup of $G$  (see for example \cite{DS,Erf,nath}).
In particular, it was proved in  \cite{DS}  that if  $H$ is a subgroup of a finite group $G$ and $\pr(H,G)\geq\epsilon>0$, 
 then there is a normal subgroup $T\leq G$ and a subgroup $B\leq H$ such that the indices $|G:T|$ and $|H:B|$, and the order of the commutator subgroup $[T,B]$  are $\epsilon$-bounded. 
 
 The present article grew out of the observation that a finite group $G$ is nilpotent if and only if $\pr(P,Q)=1$ whenever $P$ and $Q$ are Sylow subgroups of $G$ of coprime orders. Here we consider the situation where for any primes $p,q\in\pi(G)$
  with $p\ne q$ there is a Sylow $p$-subgroup $P$ and a Sylow $q$-subgroup $Q$ of $G$ such that $\pr(P,Q) \ge \epsilon$. At the start of the work it was natural to expect that the hypotheses imply that the group $G$ is ``nearly nilpotent" in some sense. In particular, we were hoping to prove that the index of the Fitting subgroup $|G:F(G)|$ is $\ep$-bounded. However, this turned out to be false (see the example in Section 4). On the other hand, we were able to establish the following theorem. Recall that $F_i(G)$ stands for the $i$th term of the upper Fitting series of the group $G$.

\begin{theorem}\label{main1}
 Let $G$ be a finite group, and let  $\epsilon>0$ be a real number such that for any distinct primes  $p,q\in\pi(G)$  
 there is a Sylow $p$-subgroup $P$ and a Sylow $q$-subgroup $Q$ of $G$ such that $\pr(P,Q) \ge \epsilon$.
Then $F_2(G)$ has $\epsilon$-bounded index in $G$.
\end{theorem}

As a related development, we studied finite groups with the property that $\pr(P,H) \ge \epsilon$ whenever $P$ is a Sylow $p$-subgroup while $H$ is a Hall $p'$-subgroup. To guarantee the existence of the Hall subgroups we focused on soluble groups. As expected, the groups in question turned out to be nearly nilpotent in the sense that the Fitting subgroup has bounded index.

\begin{theorem}\label{main2}
Let $G$ be a finite soluble group, and let  $\epsilon>0$ be a real number such that for any prime $p\in\pi(G)$ there is a Sylow $p$-subgroup $P$ and a Hall $p'$-subgroup $H$ of $G$ such that $\pr(P,H) \ge \epsilon$. Then $F(G)$ has $\epsilon$-bounded index in $G$.
\end{theorem}

A significant part of this paper deals with criteria for nilpotency and solubility of a finite group in terms of the commuting probability of Sylow subgroups. It will be convenient to introduce the following notation. 

For  a finite group $G$ and two sets of primes $\pi_1$ and $\pi_2$ we define  
\[ \prp^*_G (\pi_1, \pi_2) \]
 to be the maximum real number $\ep$ with the property that for every pair of distinct primes $p \in\pi_1$ and $q\in \pi_2$ there exists a  Sylow $p$-subgroup $P$ and a Sylow $q$-subgroup $Q$ of $G$ such that $\pr (P,Q) \ge \epsilon$. Here and throughout, we slightly abuse the standard terminology and assume that $G$ has trivial Sylow $p$-subgroup whenever $p$ does not divide $|G|$.

  If $\pi_1=\{p\}$, for short we will   write $\prp^*_G ( p, \pi_2)$.
  If the set $\pi$ contains the set of prime divisors of $|G|$, that is $\pi(G)\subseteq\pi$, we write 
$\prp^* (G)$ in place of  $\prp^*_G (\pi, \pi)$.

 Theorem \ref{main1} can now be restated in the following form:
 {\it If $G$ is a finite group such that $\prp^*(G) \geq\ep$, then $|G:F_2(G)|$ is $\epsilon$-bounded}.

Note  that for the symmetric group $\Sym_3$ and the alternating group $\A_5$ we have $\prp^* (\Sym_3)=2/3$ and $\prp^* (\A_5)=2/5.$
 Our main results on nilpotency and solubility are as follows.
\begin{theorem}\label{main3}
Let $G$ be a finite group. \begin{enumerate}
\item If $p_1, p_2$ are the smallest prime divisors of $|G|$ and $\prp^*(G)> (p_1 + p_2 - 1)/(p_1p_2)$, then G is nilpotent. In particular, if $\prp^*(G)> 2/3$, then G is nilpotent.
\item If $\prp^*(G)> 2/5$, then G is soluble.
\end{enumerate}
\end{theorem}

 The following sufficient condition for the solubility of $G$ involves only the set of odd primes. 

\begin{theorem}\label{main-odd}
Let $G$ be a finite group. If $\prp^*_G (2', 2')>7/15$, then $G$ is soluble. 
\end{theorem}

In some cases the solubility of $G$ can be detected using only the probabilities $\Pr(P,Q)$, where $P$ is a fixed Sylow 2-subgroup and $Q$ ranges over Sylow subgroups of odd order of $G$.

\begin{theorem}\label{main-2-2'} Let $G$ be a finite group. 
 If $\prp^*_G(2, 2')>2/5$, then $G$ is soluble.
\end{theorem}

The proof of the above theorem does not use the classification of finite simple groups. In fact, the use of the classification enables us to establish the following stronger result.

\begin{theorem}\label{main-2-q} Let $G$ be a finite group. 
 \begin{enumerate}
 \item If $\prp^*_{G}(2,5')>1/2$, then $G$ is soluble.
 \item If $\prp^*_G(2, 7')>5/12$, then $G$ is soluble.
 \item    If $\prp^*_G (2, p') >2/5$ for a prime $p\neq 5,7$, then $G$ is soluble.
\end{enumerate}
\end{theorem}
We remark that Theorem \ref{main-2-q} is the only result in this paper whose proof requires the classification of finite simple groups.

We also give some sufficient conditions for a Sylow subgroup of $G$ to be contained in the soluble radical (see Propositions \ref{3sol} and \ref{psol} in Section 3). 
 Examples show that the numerical bounds given in Theorems \ref{main3},  \ref{main-odd}, \ref{main-2-2'} and  \ref{main-2-q} are sharp.

In the next section we establish some auxiliary results on the commuting probability for Sylow subgroups. Section 3 contains proofs of 
Theorems \ref{main3} - \ref{main-2-q}, 
 as well as some related results. In Section 4 we provide proofs of Theorems \ref{main1} and \ref{main2}.

\section{General comments on commuting probabilities}

If $G$ is a finite group and $X,Y$ are subsets of $G$, we write $\pr(X,Y)$ for the probability that two random elements $x\in X$ and $y\in Y$ commute. Thus,
\[ \pr(X,Y) =\frac{ | \{ (x,y) \in X\times Y \mid xy=yx \} |}{|X|\,|Y|}.\]

Note that 
 $ \pr(X,Y) = \pr(Y,X)$ and 
\[   \pr(X,Y) =\frac 1 { |Y|} \sum_{y\in Y}\frac{|C_{X}(y)| }{|X| }= \frac 1 {|X|} \sum_{x\in X} \frac{|C_{Y}(x)| }{|Y| },\]
where, as usual, $C_{Y}(x)$ denotes the set of all elements of $Y$ centralizing $x$.

The next lemma is useful when considering quotient groups. Essentially, this is a variant of Lemma 2.3 in \cite{DS-commprob}. 
\begin{lemma}\label{quot}  Let $N$ be a normal subgroup of a finite group 
$G$, and let $H,K\leq G$. Then $$\pr(H,K)\leq \pr(HN/N,KN/N)\pr(N\cap H,N\cap K).$$
\end{lemma}
\begin{proof}  Let $\bar{G}=G/N$, $\bar{H}=HN/N$ and $\bar{K}=KN/N$. Write 
$\bar{H_0}$ for the set of cosets $(N\cap H)h$ with $h\in H$. If $S_0=(N\cap H)h\in\bar{H_0}$, write $S$ for the coset $Nh\in\bar{H}$. Of course, we have a natural one-to-one correspondence between $\bar{H_0}$ and $\bar{H}$.

Write
$$|H||K|\pr(H,K)=\sum_{x\in H}|C_K(x)|=\sum_{S_0\in\bar{H_0}}\sum_{x\in S_0}\frac{|NC_K(x)|}{|N|}|C_{N\cap K}(x)| $$

$$\leq\sum_{S_0\in\bar{H_0}}\sum_{x\in S_0}|C_{\bar{K}}(Nx)||C_{N\cap K}(x)|=\sum_{S\in\bar{H}}|C_{\bar{K}}(S)|\sum_{x\in S_0}|C_{N\cap K}(x)|=  $$

$$=\sum_{S\in\bar{H}}|C_{\bar{K}}(S)|\sum_{y\in N\cap K}|C_{S_0}(y)|.$$

Let $y \in N\cap K$ such that $|C_{S_0}(y)|\neq 0$ and choose $y_0\in C_{S_0}(y)$. So $S_0=(N\cap H)y_0$ and 
$$C_{S_0}(y)=(N\cap H)y_0\cap C_H(y)=C_{N\cap H}(y)y_0$$
whence $|C_{S_0}(y)|=|C_{N\cap H}(y)|.$ 
 Therefore $$|H||K|\pr(H,K)\leq \sum_{S\in\bar{H}}|C_{\bar{K}}(S)|\sum_{y\in N\cap K}|C_{N\cap H}(y)|.$$ Observe that $$\sum_{S\in\bar{H}}|C_{\bar{K}}(S)|=\frac{|H|}{|N\cap H|}\frac{|K|}{|N\cap K|}\pr(\bar{H},\bar{K})$$ and $$\sum_{y\in N\cap K}|C_{N\cap H}(y)|=|N\cap H||N\cap K|\pr(N\cap H,N\cap K).$$
It follows that $\pr(H,K)\leq \pr(\bar{H},\bar{K})\pr(N\cap H,N\cap K)$, as required.
\end{proof}

\begin{lemma}\label{lem:trivial} 
 Let $G$ be a finite group and let $H, K$ be subgroups of $G$.  Then 
\begin{enumerate}
\item If $N$ is a normal subgroup of $G$, then $\pr (HN/N,KN/N) \ge \pr (H,K)$.
\item If $H_0 \le H$, then $\pr (H_0,K) \ge \pr (H,K)$. 
\item If $G=G_1 \times G_2$, $H_i \le G_i$   and $K_i \le G_i$, then 
\[\pr (H_1 \times H_2, K_1 \times K_2) = \pr (H_1, K_1) \pr(H_2, K_2).\] 
\end{enumerate}
\end{lemma}
\begin{proof} 
The first statement follows immediately from Lemma \ref{quot}, as $\pr(N\cap K,N\cap H)\le 1$.

For the second statement note that 
\begin{eqnarray*}
 \pr(H,K) &=&\frac 1 { |K|} \sum_{y\in K}\frac{|C_{H}(y)| }{|H| }=\frac 1 { |K|} \sum_{y\in K}\frac{1 }{|y^H| } \\
&\le& \frac 1 { |K|} \sum_{y\in K}\frac{1 }{|y^{H_0}|} =\pr (H_0,K),
\end{eqnarray*}
where $y^H$ (resp. $y^{H_0}$) is the set of all $H$-conjugates (resp $H_0$-conjugates) of $y$. 

For the third statement,  
  for every $(y_1,y_2) \in K_1 \times K_2$ we have that 
 $C_{H_1 \times H_2}((y_1,y_2) ) = C_{H_1}(y_1) \times C_{H_2}(y_2)$, whence 
\begin{eqnarray*}
 \pr(H_1 \times H_2, K_1 \times K_2)  &=&\frac 1 { |K_1 \times K_2|} \sum_{(y_1,y_2)\in K_1 \times K_2}\frac{|C_{H_1 \times H_2}((y_1,y_2))| }{|H_1 \times H_2| }\\
&=&  \frac 1 { |K_1||K_2|} \sum_{(y_1,y_2) \in K_1 \times K_2} \frac{|C_{H_1}(y_1) \times C_{H_2}(y_2)| }{|H_1||H_2| } \\
&=&  \pr (H_1, K_1) \pr(H_2, K_2).
\end{eqnarray*}
\end{proof}
 
\begin{remark}\label{nopq}  
If a  finite group $G$ has no  elements of order $pq$, then for every  Sylow $p$-subgroup $P$ and every Sylow  $q$-subgroup $Q$ of $G$ we have 
 \[ \pr(P,Q) = \frac{ p^\alpha+q^\beta-1}{p^\alpha q^\beta}, \]
 where $p^\alpha=|P|$ and $q^\beta=|Q|$.  \end{remark}
 Indeed, $C_Q(x)=1$ for every $1 \neq x \in P$. Hence, 
\begin{eqnarray*}
 \pr(P,Q) &=&\frac 1 { |P| |Q|} \sum_{x \in P}{|C_{Q}(x)| } 
 = \frac   {|Q| + (|P|-1) } { |P| |Q|}.
\end{eqnarray*}

The above remark can be generalized as follows. 

\begin{lemma}\label{lem:remark}
Let $H$ and $K$ be subgroups of a finite group $G$. 
 Then
\[ \pr (H,K) \le \frac{n+m-1}{nm}, \]
 where $n=|H:C_H(K)| $ and $m=\min_{x \in H \setminus C_H(K)} |K:C_K(x)|.$
\end{lemma}
\begin{proof} 
We have $|C_K(x)| \le |K|/m$ for every $x \in H \setminus C_H(K)$, hence
\begin{eqnarray*}
 \pr(H,K) &=&\frac 1 { |K| |H|} \sum_{x \in H}{|C_{K}(x)| }\\
 &=& \frac 1 { |K | |H| } \left( \sum_{x \in C_H(K)}{|C_{K}(x)| }  + \sum_{x \in H\setminus C_H(K)} |C_{K}(x)|  \right) \\
  &\le& \frac 1 { |K | |H| } \left( |C_H(K)| |K|   + (|H| - |C_H(K)| )  \frac{|K|}{m}  \right) \\
    &=& \frac  { (m-1) |C_H(K)|   + |H| }  { |H| m}\\
    &=& \frac  { (m-1)    +n }  { n m},
\end{eqnarray*}
 as claimed.
\end{proof}

We will later require the following fact, which can be established by a straightforward calculation.

\begin{remark}\label{xy} If $x, y \geq 1$ are positive integers, then
	$$\frac{(x+1)+y-1}{(x+1)y}\leq \frac{x+y-1}{xy}.
	$$ 
\end{remark}

\begin{lemma}\label{lem:easy}
Let $P$ be a  Sylow $p$-subgroup and $Q$ be a Sylow $q$-subgroup of a finite group $G$. 
 \begin{enumerate}
 		\item If $|P:C_P(Q)| \ge p^a$, then 
		  \[ \pr (P,Q) \le \frac{p^a+q-1}{p^a q}. \]
		\item   If $[P,Q] \neq 1$, then  \[ \pr (P,Q) \le \frac{p+q-1}{pq}. \]
\end{enumerate}
\end{lemma}
\begin{proof} 
Let  $n=|P:C_P(Q)| $ and $m=\min_{x \in P \setminus C_P(Q)} |Q:C_Q(x)|.$ 
If $n \ge p^a > 1$, then  $m \ge q$. 
So, by  Lemma \ref{lem:remark} and Remark \ref{xy} we conclude that 
$\pr (P,Q) \le   ({n+m-1})/({nm}) \le ({p^a+q-1})/({p^aq}).$ 
Moreover, if $[P,Q] \neq 1$, then $|P:C_P(Q)| \ge p$. 
\end{proof}

A proof of the next lemma can be found in Eberhard \cite[Lemma 2.1]{eberhard}. 

\begin{lemma}\label{lem} 
Let $G$ be a finite group and $X$ a symmetric subset of $G$ containing the identity. If $(r+1)|X| > |G|$, then $\langle X \rangle= X^{3r}$.
 \end{lemma}

\begin{lemma}\label{lem:H_0}
 Let $G$ be a finite group and let $H, K$ be subgroups of $G$ with $\pr (H,K) \ge \epsilon >0 $. Then 
 there exists a subset $X$ of $H$ such that for  $H_0=\langle X\rangle$ the following holds.  
 \begin{enumerate}
\item $|H:H_0| \le 2/\epsilon-1$; 
\item $|K: C_K(x)| \le 2/\epsilon$ for every $x \in X$;
\item 
  $|K: C_K(x)| \le (2/\epsilon)^{6/\epsilon}$  for every $x \in H_0$.
\end{enumerate}
\end{lemma}
\begin{proof} 
Set
 \[X=\{x\in H \mid  |x^K|\leq 2/\epsilon\}.\]
 By hypotheses,  $|K: C_K(x)| \le 2/\epsilon$ for every $x \in X$. 
Note that $H \setminus X=\{ x\in H \mid  |C_K(x)| < (\epsilon/2) |K| \}$, 
whence 
\begin{eqnarray*} 
\epsilon |H||K|  &\le  & \sum_{x\in H}|C_K(x)|\le \sum_{x \in X} |K| + \sum_{x  \in H \setminus X} \frac{\epsilon}{2} |K|\\
&\le & |X| |K| +(|H| - |X|)\frac{\epsilon}{2}|K| = \left(1-\frac{\epsilon}{2}\right) |X| |K| +\frac{\epsilon}{2} |H| |K|.
\end{eqnarray*}
Therefore 
$ \epsilon |H| \le   \left(1-{\epsilon}/{2}\right)  |X|+  ({\epsilon}/{2}) |H| $
and $({\epsilon}/{2}) |H| \le   \left(1-{\epsilon}/{2}\right)  |X|$. 
Clearly, $|H_0| \ge |X| \ge \frac{{\epsilon}/{2}}{1-{\epsilon}/{2}}  |H|$ and so the index of 
 $H_0$ in $H$ is at most $2/\epsilon -1$. 
 Moreover, it follows from Lemma 
\ref{lem} that every element of $H_0$ is a product of at most $6/\epsilon$ elements of $X$.
Therefore $|h^K| \le (2/\epsilon)^{6/\epsilon}$ for every $h \in H_0$.
\end{proof}

Let a group $A$ act by automorphisms on a group $G$. We say for short that the group $A$ acts on a group $G$ coprimely if the orders of $A $ and $G$ are
finite and coprime, that is, $(|A|, |G|) = 1$. The following lemma records some standard results on coprime action. In the sequel the lemma will often be used without explicit references.

\begin{lemma} \cite[(24.1),(24.5),(24.6)]{asch} \label{coprime} Let a group $A$ act coprimely on a finite group $G$ and let $\frat(G)$ be the Frattini subgroup of $G$. Then:
\begin{enumerate}
\item If $N$ is a normal $A$-invariant subgroup of $G$, then $C_{G/N}(A)=C_G(A)N/N$.
\item $[G, A, A] = [G, A].$
\item If $G$ is abelian, then $G=[G,A]\times C_G(A).$
\item If $A$ centralizes $G/\frat(G)$, then $A$ centralizes $G$.
\end{enumerate} \end{lemma}

\begin{lemma}\label{cor:easy}
	Let $G$ be a finite group,  $P$ a   $p$-subgroup and $Q$ a  $q$-subgroup of $G$ for two different primes $p$ and $q$. 
	 Assume that $\pr (P,Q)\ge \epsilon>0$.
	\begin{enumerate}
		\item If  $P$ normalizes $Q$ and $p > (2/\epsilon)^{6/\epsilon}$, then $[P,Q]=1$. 
		\item If $p > (2/\epsilon)^{6/\epsilon}$, then $Q$ has a normal subgroup $Q_0$ such that $|Q:Q_0| \le \lfloor 2/\epsilon \rfloor!$ and $[P,Q_0]=1$. 
	\end{enumerate}
\end{lemma} 

\begin{proof} 
 Set $l= (2/\epsilon)^{6/\epsilon}$. 
	By Lemma \ref{lem:H_0}, 
	$P$ has a subgroup $P_0=\langle X\rangle$ of index at most  $2/\epsilon-1$ such that   $|Q: C_Q(x)| \le 2/\epsilon $ for every $x \in X$. Since $p>l > 2/\epsilon-1$, we deduce that $P=P_0$ and so $|Q: C_Q(x)| \le l $ for every $x \in P$. 
	
	Assume now that  $P$ normalizes $Q$ and let $x$ be an element of $P$. 
	Consider the quotient group $\bar Q= Q/\frat(Q)$ and set $A=\langle x \rangle$. Then $A$ acts coprimely on the abelian group $\bar Q$.	Since $A$ is a $p$-group  acting on $[\bar Q,A]$ and $p > l\ge |[\bar Q,A]|$,  we deduce that the action is trivial. So $[\bar Q,A,A]=[\bar Q,A]=1$. Therefore it follows from Lemma \ref{coprime} that $[Q,A]=1$ as well.
	Since this holds for every element $x\in P$, we conclude that $[P,Q]=1$. This proves (1).
	
	To prove (2),  let again $p > l= (2/\epsilon)^{6/\epsilon}$ and note that $\pr(Q,P)=\pr (P,Q)\ge \ep$. So, by Lemma \ref{lem:H_0}, 
	$Q$ has a subgroup $H_0$ of index at most  $2/\epsilon$ such that   $|P: C_P(y)| \le l$ for every $y \in H_0$.  Since  $p > l$, it follows that $|P: C_P(y)| =1$  for every $y \in H_0$, whence  	$[P, H_0]=1$. Observe that $Q_0=\cap_{x\in Q}H_0^x$ 
	is a normal subgroup of index at most $\lfloor 2/\epsilon \rfloor!$ in $Q$ with the desired properties. 
\end{proof}

Let $\prp^*(G) \geq\ep$. Since any two Sylow $p$-subgroups of $G$ are conjugate, it follows that for any $p,q\in \pi(G)$ with $q\ne p$ and any Sylow $p$-subgroup $P$ of $G$ there exists a Sylow $q$-subgroup $Q$ of $G$ such that $\pr(P,Q) \ge \epsilon$.
 
The following example shows that a group $G$ may satisfy the assumption that $\prp^*(G)\geq\ep$ and yet have a  Sylow $p$-subgroup $P$ and a Sylow $q$-subgroup $Q$, for distinct primes $p,q\in \pi(G)$, such that $\pr (P,Q)<\epsilon$. Thus the assumption that $\prp^*(G)\geq\ep$ is weaker than the assumption that $\pr (P,Q) \ge \epsilon$ for every pair of Sylow subgroups $P$ and $Q $ of $G$ corresponding to different primes.

\begin{example} {\rm For a given pair $(p,q)$, let $\Omega_{p,q}(G)$ be the set of the numbers ${\rm{Pr}}(P,Q)$ where $P$ and $Q$ are arbitrary Sylow $p$-subgroup and Sylow $q$-subgroup of $G$. Consider the symmetric group $G=\Sym_5$. 
  We can compute  $\pr (P,Q)$ (see Remark \ref{nopq}) and we find 
$$\Omega_{2,3}(G)=\left\{\frac{5}{12},\frac{1}{2}\right\},
 \Omega_{2,5}(G)=\left\{\frac{3}{10}\right\}, 
\Omega_{3,5}(G)=\left\{\frac{7}{15}\right\}.
$$ 
Although $|\Omega_{2,3}(G)|\neq 1$, the group $G$ is not the example we are looking for: since $\frac{3}{10} < \frac{5}{12} < \frac{7}{15} < \frac{1}{2},$ 
 we have $\prp^*(G)=\frac{3}{10}$ and so $\pr (P,Q) \ge \frac{3}{10}$ for every Sylow $2$-subgroup $P$ and every Sylow $3$-subgroup $Q$.  

Let $X=\Sym_5 \times \Sym_3^t$, where $\Sym_3^t$ denotes the direct product of $t$ copies of $\Sym_3$.  Note that  $\Omega_{2,3}( \Sym_3)=\left\{\frac{2}{3}\right\}$.
 It follows from Lemma 2.2 (3), that 
$$\Omega_{2,3}(X)=\left\{\frac{5}{12} \left(\frac{2}{3}\right)^t,\frac{1}{2}\left(\frac{2}{3}\right)^t\right\}, \
\Omega_{2,5}(X)=\left\{\frac{3}{10}\right\}, \ 
\Omega_{3,5}(X)=\left\{\frac{7}{15}\right\}.
$$ 
In particular, if $\frac{1}{2} \left(\frac{2}{3}\right)^t\leq \frac{3}{10}$, that is, if $t\geq 2$, then 
$ \prp^*(X)= \frac{1}{2} \left(\frac{2}{3}\right)^t$, 
 but there exist a Sylow 2-subgroup $P$ and a Sylow 3-subgroup $Q$ of $X$ with $\pr(P,Q)=\frac{5}{12} \left(\frac{2}{3}\right)^t<\prp^*(X).$}
\end{example}

The following lemma is a straightforward consequence of Lemma  \ref{lem:trivial}. It will be used without explicit mention.

\begin{lemma}\label{lem:trivial*} 
 Let $G$ be a finite group such that $\prp^*_G( \pi_1, \pi_2) \geq\ep$, for some sets of primes  $\pi_1, \pi_2$. 
If $N$ is a normal subgroup of $G$, then $\prp^*_{G/N} ( \pi_1, \pi_2) \geq\ep$ and $\prp^*_N( \pi_1, \pi_2) \geq\ep$. 
\end{lemma}
\begin{proof} The lemma follows from Lemma \ref{lem:trivial} and the fact that if $P$ is a Sylow subgroup  of $G$, then $PN/N$ and $P\cap N$ are Sylow subgroups of  $G/N$ and $N$, respectively. 
\end{proof}

\section{Criteria for nilpotency and solubility} 

In this section we establish Theorems \ref{main3}, \ref{main-odd}, \ref{main-2-2'} and \ref{main-2-q}, as well as some related results. Throughout, we use the fact that groups of odd order are soluble \cite{fetho} without explicit references.
 
Note that for the symmetric group $\Sym_3$ we have $\prp^* (\Sym_3)=2/3$. The following proposition shows that when  $\prp^*(G)>2/3$ the group is necessarily nilpotent.

\begin{proposition}\label{cor:nilpotent} Let $G$ be a  finite  group such that $\prp^*(G)>2/3$. Then $G$ is nilpotent. 
\end{proposition}
\begin{proof} 
 By hypothesis, for every two primes $q \neq p$ there exists  a Sylow $p$-subgroup $P$ and  a Sylow  $q$-subgroup $Q$ of $G$ such that 
  $\pr (P,Q) > 2/3$. As all Sylow $q$-subgroups are conjugate, here we can assume that $P$ is fixed.  
  If $[P,Q] \neq1$, then by  Lemma \ref{lem:easy} (2) and Remark \ref{xy} 
 we have 
\[ \pr(P,Q) \le \frac{p+q-1}{pq} \le \frac{2+3-1}{2 \cdot 3}=\frac 2{3},  \]
 a contradiction. Therefore $[P,Q]=1$ for every $q \neq p$. 
 Then the index of $C_G(P)$ in $G$ is a $p$-number and thus $G=PC_G(P)$. 
 
 This proves that every Sylow $p$-subgroup of $G$ is normal and thus $G$ is nilpotent. 
\end{proof}

With  more technical arguments we can sharpen the previous result. 

\begin{proposition}\label{nice} Let $G$ be a finite group and $p$ a prime divisor of $|G|$. Suppose that 
  $\prp^*_G (p, p')> \frac{p+q-1}{pq}$, where $q$ is the smallest prime in $\pi(G)\setminus\{p\}$. Then
$G={\rm{O}}_p(G) \times {\rm{O}}_{p^\prime}(G)$.
\end{proposition}
\begin{proof} 
Let us fix a Sylow $p$-subgroup $P$ of $G$.    Let $r\neq p$ and let $R$ be a Sylow $r$-subgroup of $G$ such that 
\[ \pr(R,P) > \frac{p+q-1}{pq}.\] 
If $[R,P]\neq 1$, then by Lemma \ref{lem:easy} (2) and Remark \ref{xy} 
 we have 
\[ \pr(R,P) \le \frac{p+r-1}{pr}  \le \frac{p+q-1}{pq},   \]
       a contradiction. 

This implies
that $C_G(P)$ contains a Sylow $r$-subgroup of $G$ for every $r\neq p$.
We deduce that $G=PC_G(P)$. By the Schur-Zassenhaus theorem (see Theorem 3.8 in \cite{Isaacs}), $C_G(P)$
contains a $p$-complement $H$ and therefore $G= P\times H.$
     \end{proof} 
     
     Theorem \ref{main3} (1) is now straightforward.
\begin{proposition}\label{pq}
Let $p_1, p_2$ be the smallest prime divisors of $|G|$ and suppose that $\prp^*(G) > (p_1 + p_2 - 1)/(p_1p_2)$. Then G is nilpotent.
	\end{proposition}
\begin{proof}
Remark \ref{xy} implies that $(p + q - 1)/pq  \ge (p_1 + p_2 - 1)/p_1p_2$
  for every pair of distinct prime divisors $p$ and $q$ of $|G|.$   
Now let $p$ be a prime divisor of $G$. Then the assumptions of Proposition \ref{nice} are satisfied and so $G={\rm{O}}_{p}(G) \times {\rm{O}}_{p\prime}(G)$. This holds for every prime divisor of $|G|$, therefore the group $G$ is nilpotent.
\end{proof}

We remark that Proposition \ref{cor:nilpotent} is actually a corollary of Proposition \ref{pq}, since $(2+3-1)/6=2/3$. 
Furthermore,  Proposition \ref{cor:nilpotent} implies that if $G$ is a  finite  group  such that $\prp^*(G) >1/2$, then for $\pi=\{ p \in \pi(G) \mid p \ge 1/(2 \prp^*(G) -1) \}$ we have $G= O_\pi (G) \times O_{\pi'}(G)$  and moreover $O_\pi(G)$ is nilpotent. 

\begin{example} \label{A5} {\rm For the alternating group $\A_5$ we have $\prp^* (\A_5)=2/5$. 
 Indeed, if $P_p$ denotes a  Sylow $p$-subgroup of $\A_5$,  by  Remark \ref{nopq} we have 
 \[ \Pr(P_2,P_3)=1/2 
 \quad \Pr(P_2,P_5)=2/5 
  \quad \Pr(P_3,P_5)=7/15. 
   \]
	}
  \end{example}
	
 The following proposition establishes  Theorem \ref{main3} (2). 

\begin{proposition}\label{cor:soluble} Let $G$ be a  finite  group  such that $\prp^*(G)>2/5$. Then $G$ is  soluble. 
\end{proposition}
\begin{proof} Let $G$ be a counterexample of minimal order. As the hypothesis is inherited by quotients and normal subgroups, $G$ is a nonabelian finite simple group. 
	
Let $q\geq 5$ be a prime divisor of $|G|$ and let $Q$ be a Sylow $q$-subgroup of $G$. 
 For every prime $p$ with $p\neq q$,  let $P$ be a  Sylow $p$-subgroup of $G$ such that $\pr(P,Q)>2/5$. 
If   $|P:C_P(Q)| \ge 4 $, then,  by Remark \ref{xy} and Lemma \ref{lem:easy},  we have 
		  \[  \pr (P,Q) \le \frac{4+5-1}{4 \cdot 5} =\frac {2}{5} , \]
 a contradiction.  
 So  $|P:C_P(Q)|  <4$ for every prime $p \neq q$. In particular, $C_G(Q)$ contains a Sylow $p$-subgroup for every prime $p >5$, a subgroup of index at most  $2$ in a Sylow $2$-subgroup and  a subgroup of index  at most  $3$ in a Sylow $3$-subgroup. We conclude that the index of the subgroup $H=QC_G(Q)$ divides $6$. 
 
 Since $G$ is simple, $H$ is a proper subgroup of $G$ and $G$ is isomorphic to a subgroup of $S_6,$ as it acts faithfully by permuting the cosets of $H$. The only possibilities are $G=\A_5$ and $G=\A_6.$ But this leads to a contradiction because
  $\prp^* (\A_5)=2/5$ (see Example \ref{A5}) and if
  $G=\A_6$, then, by  Remark \ref{nopq},  $\pr(P,Q)=(8+9-1)/(8 \cdot 9)=2/9 <2/5$ for any Sylow $2$-subgroup $P$ and any Sylow $3$-subgroup $Q$ of $G$.
 	\end{proof}
	
Applying less elementary arguments, we will now show that the solubility of $G$ can be detected by checking only some subsets of $\pi(G)$. 

In 1904  Burnside proved a theorem on the solubility of a group of order $p^\alpha q^\beta,$ which
became an important tool in the theory of finite groups. Burnside's theorem is based on the following
lemma on nonsimplicity, which is known as Burnside's $p^\alpha$-Lemma (see e.g. Theorem 3.9 in \cite{characters}). 

\begin{lemma}[Burnside's $p^\alpha$-Lemma] \label{burnsideLemma}
If the size of a conjugacy class of a finite group $G$  is a power of a prime number, then G is not a simple group.
\end{lemma} 

We can now prove Theorem \ref{main-odd}, which we restate here.
\medskip

{\sc Theorem \ref{main-odd}}. {\it Let $G$ be a finite group. If $\prp^*_G (2', 2')>7/15$, then $G$ is soluble. }
\begin{proof}
	Let $G$ be a counterexample of minimal order. As the hypothesis is inherited by quotients and normal subgroups, $G$ is a nonabelian simple group. 
	
	Let $q \neq 2$ be the smallest odd prime divisor of $|G|$ and let  $Q$ be a Sylow $q$-subgroup of $G$. 
	 For every prime $p \neq 2, q$ let $P$ be a Sylow $p$-subgroup of $G$ such that 
	$\pr(P,Q) >7/15$. Since $p \ge 5$ and $q \ge 3$, if $[P, Q] \neq 1$, then  by Remark \ref{xy} and Lemma \ref{lem:easy}  we have that 
		  \[ \frac{7}{15} < \pr (P,Q) \le \frac{3+5-1}{3 \cdot 5} =\frac {7}{15} , \]
		  a contradiction. 
	Therefore   $C_G(Q)$ contains a Sylow $p$-subgroup of $G$ for every $p \neq 2, q$. In particular $|G:QC_G(Q)|$ is a 2-power. 
		But then for $1\neq z \in Z(Q)$ the index $|G:C_G(z)|$ is a $2$-power as well. 
	In view of Lemma \ref{burnsideLemma}, this is a contradiction. 
\end{proof} 

We now proceed to establish some results which are needed in the proofs of Theorems \ref{main-2-2'} and \ref{main-2-q}. 

\begin{proposition}\label{2-q}  
Let $G$ be a finite group.
 \begin{enumerate}
\item  If $p$ is a prime and $\prp^*_G(2,p')>1/2$, 
 then $G$ is soluble.
\item  If  $\prp^*_G(2, 3')>2/5$, 
 then $G$ is soluble.
\end{enumerate}
\end{proposition}
\begin{proof} 
	Let $G$ be a counterexample of minimal order to the first statement. As the hypothesis is inherited by quotients and normal subgroups, $G$ is a nonabelian simple group. 
	
	Let $P$ be a Sylow $2$-subgroup of $G$. 
 For every prime $q$ such that $q\neq 2, p$  let $Q$ be a Sylow $q$-subgroup of $G$ such that $\pr(P,Q)>1/2$. 
If   $|P:C_P(Q)| \ge 4 $, then  by Remark \ref{xy} and Lemma \ref{lem:easy}  we have that 
\begin{equation}\label{1/2}
		   \pr (P,Q) \le \frac{4+3-1}{4 \cdot 3} =\frac {1}{2} , 
\end{equation}   
a contradiction.  
 So  
 \[|P:C_P(Q)|  \le 2,\]
  and in particular $C_P(Q)$ contains the Frattini subgroup $\Phi=\frat(P)$ of $P$, for every prime $q \neq 2, p$. 
   Then $C_G(\Phi)$ contains a Sylow $q$-subgroup of $G$ for every $q  \neq 2, p$. 
   
   If $\Phi \neq 1$ and $1\neq z \in Z(P)\cap \Phi$, it follows that $|G:C_G(z)|$ is a $p$-power. Then, by Lemma \ref{burnsideLemma}, $G$ is not simple, a contradiction. 
    This proves that   $\Phi=1$ and $P$ is elementary abelian.  
   In particular, as $C_P(Q) \neq 1$ and  a Sylow $2$-subgroup of a nonabelian finite simple group has order at least $4$ (see for instance Corollary 5.14 in \cite{Isaacs}), for every $q\neq p$ there is an involution $z_q\in P$ centralizing a Sylow $q$-subgroup of $G$.
	
 Note that in a finite simple group with elementary abelian Sylow $2$-subgroups 
 	all involutions are conjugate (see  \cite{walter} or \cite[2.2]{bender}). Thus, their centralizers are conjugate and in particular have the same order. We know that the  involution $z_q$ centralizes a Sylow $q$-subgroup of $G$. In particular, $q$ does not divide the index of  $C_G(z_q)$ in $G$. This happens for every $q\neq p$. Taking into account that the involutions are conjugate, we deduce that the index of the centralizer of any involution is a power of $p$. In view of Lemma \ref{burnsideLemma} this is a contradiction.

The second statement follows by the same arguments using the fact that if $p=3$, then for every odd $q\neq p$ the inequality \eqref{1/2} takes the form
\begin{equation*}
		   \pr (P,Q) \le \frac{4+5-1}{4 \cdot 5} =\frac {2}{5}.
\end{equation*}   
\end{proof} 

One immediate corollary of Proposition \ref{2-q}  is  Theorem \ref{main-2-2'}. 
\medskip

{\sc Theorem \ref{main-2-2'}}. {\it Let $G$ be a finite group. 
  If $\prp^*_G(2, 2')>2/5$, then $G$ is soluble.}
  \begin{proof}
Note that if $\prp^*_G(2, 2')>2/5$, then in particular $\prp^*_G(2, 3')>2/5$, thus the result follows follows from Proposition \ref{2-q} (2). 
\end{proof}

We remark that  in $\A_5$, as mentioned above,
\[ \Pr(P_3,P_5)=7/15, 
 \quad \Pr(P_2,P_5)=2/5, 
  \quad \Pr(P_2,P_3)=1/2, 
   \]
 for  every Sylow 3-subgroup $P_3$, every Sylow $5$-subgroup $P_5$ and every Sylow $2$-subgroup $P_2.$ 
 Thus
 \begin{eqnarray*} 
 \prp^*_{\A_5} (2',2') &=&7/15,\\
 \prp^*_{\A_5} (2,2')\, &=&2/5,  \quad \prp^*_{\A_5} (2,5')=1/2, \quad \prp^*_{\A_5} (2,3')=2/5.
 \end{eqnarray*}
 Therefore the values in Theorems \ref{main-odd} and \ref{main-2-2'}, Proposition \ref{2-q} (2)  and Proposition \ref{2-q} (1) for $p=5$,  are best possible.

Using the classification of finite simple groups we can sharpen the bound in Proposition \ref{2-q} (1)  in the case where $p \neq 5$. 

We remark that Proposition \ref{7} is the only result in which the classification of finite simple groups is used.

\begin{proposition}\label{7}  Let $G$ be a finite group and $p\neq 2$ a prime. 
 \begin{enumerate}
\item  If $p \neq 5, 7$ and $\prp^*_G(2,p') >2/5$, then $G$ is soluble.
\item  If $\prp^*_G(2,7') >5/12$, then $G$ is soluble.
\end{enumerate}
\end{proposition}
\begin{proof}
 We will first prove the following statement: 

\begin{itemize}
\item[$(**)$] If $p\neq 2,5$ and $G$ is a finite nonabelian simple group such that $\prp^*_G(2,p') >2/5$, then $p=7$ and  $G$ is  isomorphic to $\PSL(2,7)$.
\end{itemize}

	Let $\pi=\pi(G)\setminus \{2,3,p\}$ and let $P$ be a Sylow $2$-subgroup of $G$. As in Proposition \ref{2-q}, since 
	$(4+5-1)/20=2/5$, 
	it follows   from Remark \ref{xy} and Lemma \ref{lem:easy}  that  for every prime for $q\in\pi$ there exists a Sylow $q$-subgroup $Q$ of $G$  such that $|P:C_P(Q)|\leq 2.$ 
	
	Let $\Phi=\frat(P).$ If $\Phi=1,$ then the involutions of $G$ are conjugate (see  \cite{walter} or \cite[2.2]{bender}). Since
	for every $q\notin \{p,3\},$ the centralizer of a Sylow $q$-subgroup contains at least one involution, we deduce that whenever $z\in G$ is an involution, $|G:C_G(z)|=3^ap^b$ for suitable 
	 integers $a$ and $b$. 
	 If $\Phi\neq 1$, then $C_G(\Phi)$ contains a Sylow $q$-subgroup of $G$, for every $q \in \pi.$ For an involution $z \in Z(P)\cap \Phi,$ again  $|G:C_G(z)|=3^ap^b$, where $a$ and $b$ are suitable 
	  integers. Moreover, by Lemma \ref{burnsideLemma}, $p \neq 3$ and $a, b \neq 0$. 
	
	The pairs $(J,G)$, where $J$ is a subgroup of $G$ with $|G:J|$ a product of two prime-powers are listed in \cite{twop}. 
		Using the facts that $p \neq 2, 5$ and $Z(J)\neq 1$, it can be easily checked that none of the simple groups  $G$ listed in Tables 3.1 and 3.2 of \cite{twop}   contains an involution $z$ such that $|G:C_G(z)|=3^ap^b$.
		
	So we may assume that $G$ is a simple group of Lie type. Let $r$ be the characteristic of $G$, and let $R$ be a Sylow $r$-subgroup and $P$, as above, a Sylow $2$-subgroup of $G$. Since $C_P(R)=1$ (see for instance \cite[Corollary 3.1.4]{GLS}), 
	if $r\notin \{2,3,p\}$ using Lemma \ref{lem:easy}  and Remark \ref{xy} we get
	$$\pr(P,R)\leq  \frac{5+4-1}{5\cdot 4} 
	=\frac{2}{5}.$$ This contradicts the hypothesis.
	
	Similarly, if $r=3,$ then
	$$\pr(P,R)\leq \frac{|P|+2}{3|P|}.$$ Therefore $(|P|+2)/3|P|>2/5$ and this implies $|P|\leq 8.$ 
	So $16$ does not divide $|G|.$ Notice that in this case either $G= \PSL(2,3^n)$ or $G= {^2G_2}(3^{2n+1})$ (see e.g. \cite{gg} and \cite{walter}). 
	If  $G=\PSL(2,3^n)$,  then $G$ does not contain elements of order $6$ and
		for any Sylow $2$-subgroup $P$ and any Sylow $3$-subgroup $Q$, since $|P|\geq 4$ and $|Q|\geq 9,$ we have
		$$\pr(P,Q)=\frac{|P|+|Q|-1}{|P||Q|}\leq \frac{4+9-1}{36}=\frac{1}{3}< \frac{2}{5}.$$
	The group $G= {^2G_2}(3^{2n+1})$ has no subgroup of index $3^ap^b$ (see Table 5.1 of \cite{twop}). 
	
	Therefore we are reduced to the case where $G$ is a simple group of Lie type in characteristic $2$ or $p$. 

	Summarizing, we have to examine Tables 4.1, 4.2, 4.3, 4.4, 4.5, 5.1 in \cite{twop}, looking
	for pairs $(G,J)$, where:
	\begin{itemize}
\item	 $J=C_G(z)$ for some involution $z\in G$ (in particular $Z(J) \neq 1$), 
\item $|G:J|=3^ap^b$, with $a, b \geq 1$, $p\neq 5$,
\item  the characteristic of $G$ is $2$ or $p.$
\end{itemize}
	We also use the well known result by Zsigmondy \cite{z}, which says that given two integers $q \geq 2$ and $n \geq 3$, with $(q,n)\neq (2,6)$ there exists a prime $u$ (a Zsigmondy prime) such that $u$ divides $q^n-1$ but $u$ does not divide $q^i-1$ for $i<n.$
 So, if $(q^n-1)/(q-1)$ is a power of $3$, then $n=2$, since $3$ must be a Zsigmondy prime and it divides $q^2-1=(q-1)(q+1)$.	

 We give just one example of the arguments we used. In Table 4.4 we find that $\Omega_{2m+1}(q)$ has a subgroup $K=Z_2 \times \Omega_{2m}^+(q)$ of index $q^m(q^m+1)/2$ and $q^m=2 s^c-1$, for a prime $s$. 
 Here $q=p^b$ is odd and so $(q^m+1)/2=3^a$. As $3$ divides $q^2-1$ and $m \ge 3$, there exists a   Zsigmondy prime divisor $u$ of $q^{2m}-1$ with $u \neq 3$. 
   Since $u$ does not divide $q^m-1$, we have that $u \neq 3$ divides $q^m+1$, a contradiction. 
	
	 It turns out that a group $G$  in  Tables 4.1, 4.2, 4.3, 4.4, 4.5, 5.1 in \cite{twop}, may satisfy the above condition only if $G= \PSp(6,2)$ or 
	  $G= \PSL(2,q)$ with $q=8$ or $q=p$ and $(p-1)/2=3^a$ or $(p+1)/2=3^a$.
	  
	  We exclude $G= \PSp(6,2)$,  since in that case a direct calculation (e.g. with GAP \cite{GAP4}) shows that $\pr(P,Q)\leq  5/288,$ 
	   for every Sylow $2$-subgroup $P$ and every Sylow $3$-subgroup $Q$ of $G$.

	   Also, when $G=\PSL(2,8)$, then $\pr(P,Q)=2/9<2/5,$ 	   for every Sylow $2$-subgroup $P$ and every Sylow $3$-subgroup $Q$ of $G$.
	
	Finally suppose $G=\PSL(2,p),$ with $(p-1)/2=3^a$  or $(p+1)/2=3^a$. 
	
			Let $\delta \in \{-1,1\}$ be such that  $(p+\delta)/2$ is a power of $3$.
		  Then a Sylow $3$-subgroup of $G$ is cyclic of order $(p+\delta)/2$ and it is self-centralizing (see e.g. \cite[II \S8]{huppert}),
	 so  $G$ does not contain elements of order $6$. Since $(4+9-1)/36=1/3 <2/5$, we conclude that the order of a Sylow $3$-subgroup is at most $3$. 
	  Therefore $p=7$ or $p=5$. 
	 Since $p \neq 5$, we are left with the  case $p=7$. Note that in $G=\PSL(2,7)$ we have $\pr(P,Q)=5/12>2/5$
		for every Sylow $2$-subgroup $P$ and every Sylow $3$-subgroup $Q$, whence $\prp^*_G (2,q) >2/5$
		for  any prime $q\ne 7.$

This concludes the proof of statement $(**)$; i.e.  if $p\neq 2,5$ and $G$ is a finite nonabelian simple group such that $\prp^*_G (2,q) >2/5$ for  any prime $q\in\pi(G)\setminus\{p\}$, then $p=7$ and  $G$ is  isomorphic to $\PSL(2,7)$.

We now complete the proof of the proposition.

\noindent (1) Let $G$ be a counterexample of minimal order. As the hypothesis is inherited by quotients and normal subgroups, $G$ is a nonabelian simple group. It follows from   statement   $(**)$ that $p=7$, contrary to the assumption that $p\neq7$.
	
\noindent (2) Let $G$ be a counterexample of minimal order. Again $G$ is a nonabelian simple group. As $5/12>2/5$, we deduce from  statement  $(**)$ that $G=\PSL(2,7)$.  But in this case  $\pr
(P,Q)=5/12$	for every Sylow $2$-subgroup $P$ and every Sylow $3$-subgroup $Q.$ This final contradiction proves the result.
\end{proof}
 
Now Theorem \ref{main-2-q} follows easily from Propositions \ref{2-q} and \ref{7}.
\medskip

{\sc Theorem \ref{main-2-q}}. {\it Let $G$ be a finite group. 
 \begin{enumerate}
 \item If $\prp^*_{G}(2,5')>1/2$, then $G$ is soluble.
 \item If $\prp^*_G(2, 7')>5/12$, then $G$ is soluble.
 \item    If $\prp^*_G (2, p') >2/5$ for a prime $p\neq2,5,7$, then $G$ is soluble.
\end{enumerate}}

\begin{proof}
\begin{enumerate}
 \item follows from Proposition \ref{2-q} (1).
\item follows from Proposition \ref{7} (2).
\item  follows from Proposition \ref{7} (1).
\end{enumerate}
\end{proof} 

 We remark that the bounds in Theorem \ref{main-2-q} are best possible, as the examples of the groups $\A_5$ and $\PSL(2,7)$ described above show. Indeed,  $\prp^*_{\A_5}(2, 5')=1/2$, $\prp^*_{\PSL(2,7)}(2, 7')=5/12$ and $\prp^*_{\A_5}(2, p')=2/5$ for $p \neq 5, 7$. 
 	
Now we deal with the problem of detecting $p$-solubility of $G$ for a fixed prime $p \neq 2$. 

\begin{proposition} \label{3sol} Let $G$ be a finite group such that $\prp^*_G(3,3')(G)>7/15,$ then the Sylow $3$-subgroups of $G$ are contained in the soluble radical $R(G)$. In particular, $G$ is $3$-soluble.
\end{proposition}
\begin{proof} 
 Let $G$ be a counterexample of minimal order  and let $P$ be a Sylow $3$-subgroup of $G$. 

  Assume $G$ has two minimal normal subgroups $N_1$ and $N_2$. By minimality $PN_i/N_i$ is contained in the soluble radical $R_i/N_i$ of $G/N_i$ for $i=1,2$. Let $R=R_1 \cap R_2$. Note that $P \le R$ and $R$ is normal in $G$. Since $R$ is isomorphic to a subgroup of $RN_1/N_1 \times RN_2/N_2$ we deduce that $R$ is soluble. Hence $P \le R \le R(G)$, a contradiction. So $G$ has a unique minimal  normal subgroup $N$, and by minimality $N$ is not abelian. 

	By hypotheses, if 
	 $q\neq 3$ is a prime, then there exists   a Sylow $q$-subgroup $Q$ of $G$ such that 
	$\pr(P,Q) > 7/15$. 
	Assume that $q\ge 5$. If $[P,Q]\neq 1,$ 
	 then  from Remark \ref{xy} and Lemma \ref{lem:easy}   we get 
		$$\pr(P,Q)\leq \frac{3+5-1}{3 \cdot 5}= \frac{7}{15},$$
 a contradiction. Therefore $[P,Q]=1$ for every $q \ge 5$. 

This implies that $|G:PC_G(P)|$ is a $2$-power.  
In particular, $|G:C_G(z)|$ is a 2-power for every $z\in Z(P).$ It follows from the main result in \cite{kaz} that $\langle z^G\rangle$ is a soluble normal subgroup of $G,$ contradicting the fact that the unique minimal normal subgroup $N$ of $G$ is nonabelian.
\end{proof}

\begin{proposition}\label{psol} Let $G$ be a finite group and $p\geq 5$ a prime. If $\prp^*_G(p,p')>2/5,$ then the Sylow $p$-subgroups of $G$ are contained in the  soluble radical $R(G)$. In particular, $G$ is $p$-soluble.
\end{proposition}

\begin{proof} 
  Let $G$ be a counterexample of minimal order. As in the proof of Proposition \ref{3sol}, it follows that $G$ has a unique minimal normal subgroup $N$, which is nonabelian.  In particular, $C_G(N)$ is trivial.

Let $P$ be a Sylow $p$-subgroup of $G$. 
 As $(4+5-1)/20=2/5$, it follows 
 from Remark \ref{xy} and Lemma \ref{lem:easy}   
 that $|G:PC_G(P)|$ divides $6$. In particular if $z$ is an element in $Z(P)$ of order $p$, then $|G:C_G(z)|$ divides $6$. 

Then $G$  permutes the cosets of $C_G(z)$ and the kernel of the action is the normal core $K$ of $C_G(z)$ in $G$. If $K$ contains $N$, then $1\neq z \in C_G(N)$, a contradiction. Therefore, $K$ is trivial and $G$ is isomorphic to a subgroup of $S_6$. In particular, $|G|$ divides $6!$ and so $p=5=|P|$ and $P=\langle z \rangle$. 

If $z \notin N$, then $N$ is  isomorphic to  a subgroup of  $S_6$ whose order is not divisible by $5$. It follows that $N$ is soluble,  a contradiction. 

Thus, $z \in N$ and, by minimality, $N=G$ is a simple nonabelian subgroup of $S_6$, that is, $N \in \{\A_5, \A_6\}$. In both cases, the index of $P=C_N(P)$ does not divide $6$, giving the final contradiction. 
\end{proof}

\section{Theorems \ref{main1} and \ref{main2}} 

This section furnishes proofs of Theorems \ref{main1} and \ref{main2}. We start with a general comment about simple groups.
\begin{lemma}\label{lem:simple}
 Let $S$ be a nonabelian finite simple group such that $\prp^* (S) \ge \epsilon >0$.
 Then $|S|$ is $\ep$-bounded. 
\end{lemma} 
\begin{proof}  
Set $l= (2-\epsilon)/\epsilon$. Let $p$ be the largest prime divisor of $|S|$ and  $P$ a Sylow $p$-subgroup of $S$. 

 Assume first that $p > l$. 
 Let $|S| =p_1^{a_1} \cdots p_t^{a_t}$ where $p_i$ are distinct primes and $p_1=p$. 
  For each prime $p_j$ with $j \neq 1$ let $Q_j$ be a  Sylow $p_j$-subgroup  such that  $\pr (P,Q_j) \ge \epsilon.$
  
 Choose $j$ such that $[P, Q_j] \neq 1$ and let $p_j^{b_j} = |Q_j : C_{Q_j}(P)|$. By  Lemma \ref{lem:easy} we have that 
\[   \epsilon \le \pr (Q_j , P) \le \frac{p_j^{b_j} +p-1}{p_j^{b_j} p} \] 
  and this implies 
    \begin{equation}\label{p_jb} 
   p_j^{b_j}  \le \frac{p-1}{\ep p-1}= \frac 1{\ep} \left( 1 + \frac{1-\ep}{\ep p -1} \right). 
     \end{equation}
  Since $p > l= (2-\epsilon)/\epsilon$, we have $\ep p -1 > 1-\ep$ and therefore  by \eqref{p_jb} 
   \[ p_j^{b_j}  \le \frac 2{\ep}.\] 
 It follows that $[P, Q_j] = 1$ for every $p_j > 2/\ep$. 
 Therefore, for a nontrivial element $1 \neq z \in Z(P)$  the index of $C_S(z)$  in $S$ is bounded by 
   \[ \prod_{p_j < 2/\ep} p_j^{b_j} \le \prod_{p_j <2/\ep }  \frac 2{\ep} \le \left(\frac 2{\ep}\right)^{2/\ep}. \]
 As $S$ is a simple group, the normal core of $C_S(z)$ is trivial, thus $|S| \le (|S:C_S(z)|)!$ is $\ep$-bounded. 
 
 Let us now consider the case where $p \le l= (2-\epsilon)/\epsilon$. It is a well-known consequence of the classification of finite simple groups that, given a set of primes $\pi$, there are only finitely many finite nonabelian simple groups $T$ such that $\pi(T)=\pi$.
Thus, if $p \le l$, then $|S|$ is $\epsilon$-bounded. 

An alternative proof, which is independent of the classification, can be obtained as follows.

  Let again $|S| =p_1^{a_1} \cdots p_t^{a_t}$ where $p_i$ are distinct primes, and suppose that
  \[ p_t^{a_t} > p_j^{a_j},\]
  for every $j \neq t$. Since $t \le p \le l$ and $|S| \le (p_t^{a_t} )^l$, it is sufficient to bound the order $p_t^{a_t} $ of a Sylow $p_t$-subgroup $R$ to conclude that the order 
  of $S$ is $\ep$-bounded. 
  
    For each prime $p_j$ with $j \neq t$ let $Q_j$ be a  Sylow $p_j$-subgroup  such that  $\pr (R,Q_j) \ge \epsilon.$ 
    By Lemma \ref{lem:H_0}, for every $j \neq t$ there exists a subgroup $H_j$ of $R$ such that 
     \begin{enumerate}
\item $|R:H_j| \le 2/\epsilon-1$; 
\item $|Q_j: C_{Q_j}(x)| \le (2/\epsilon)^{6/\epsilon}$  for every $x \in H_j$.
\end{enumerate}
As $t \le l$, the intersection $H= \cap_{j \neq t} H_j$ has $\ep$-bounded index in $R$. If $H=1$, then $R$ has   $\ep$-bounded-order, as desired. 
 Otherwise $H\neq 1$ and we can find a non trivial element $1 \neq z \in Z(H)$. 
 Since $|Q_j: C_{Q_j}(z)| \le (2/\epsilon)^{6/\epsilon}$  for every $j \neq t$ we have that  
   \[ |S:C_S(z)| \le |R:H| \prod_{j \neq t} |Q_j: C_{Q_j}(z)| \le  |R:H| ( (2/\epsilon)^{6/\epsilon})^l. \]
  Therefore $|S:C_S(z)| $ is $\ep$-bounded. We conclude again that  $|S|$  is   $\ep$-bounded.
\end{proof}

\begin{lemma}\label{lem:nonab-power}
 Let $G=H^t$ be a direct product of $t$ copies of a non-nilpotent finite  group $H$  and assume that $\prp^* (G) \ge \epsilon >0$.
  Then $t$ is $\ep$-bounded. 
\end{lemma}
\begin{proof}  
Write $G=H_1\times\dots\times H_t$, where $H_i$ is isomorphic to $H$. 

 By Proposition \ref{cor:nilpotent} 
 there exist two primes $p \neq q$  in $\pi(H)$ such that, for every Sylow $p$-subgroup $P$  and every 
 Sylow $q$-subgroup  $Q$ of $H$, we have $\pr (P,Q) \le 2/3.$ 

Choose a Sylow $p$-subgroup $\hat P$  and a Sylow  $q$-subgroup 
$\hat Q$ of $G$. Write 
$\hat P=P_1\times\dots\times P_t$  and $\hat Q=Q_1\times\dots\times Q_t$, where $P_i, Q_i$ are Sylow subgroups of  $H_i$. 
Thus $\pr (P_i,Q_i) \le 2/3.$ 

In view of Lemma \ref{lem:trivial} (3)  we have 
\[ \ep \le \pr (\hat P, \hat Q) = \prod_{1\le i \le t} \pr (P_i,Q_i) \le (2/3)^t. \] 
So $t$ is $\ep$-bounded and the lemma follows. 
\end{proof}

\begin{corollary}\label{cor:simple-product}
 Let $G$ be a direct product of  nonabelian finite simple groups such that $\prp^* (G) \ge \epsilon >0$.
  Then $|G|$ is $\ep$-bounded. 
\end{corollary}
\begin{proof}   
Write $G=S_1^{t_i}\times\dots\times S_r^{t_r}$, where each $S_i$ is a nonabelian finite simple group. 
 Each $S_i^{t_i}$ is a homomorphic image of $G$, whence $\prp^* (S_i^{t_i})\geq\epsilon$. 
 It follows from Lemma \ref{lem:nonab-power} that each $t_i$ is $\ep$-bounded. 
 Moreover,  as $S_i$ is a homomorphic image of $S_i^{t_i}$, we observe that $\prp^* (S_i)\geq\epsilon$.  Therefore, by Lemma \ref{lem:simple} we have that $S_i$ has $\ep$-bounded order. 
 In particular, since there are at most $(n!)^{\log n} $ finite groups of order $n$, 
 there are only  $\ep$-boundedly many possible choices for $S_i$. 
  So, $r$ is $\ep$-bounded. 
  We conclude that  $G= S_1^{t_i}\times\dots\times S_r^{t_r}$ has $\ep$-bounded order. 
\end{proof}

The next useful result is an immediate consequence of Lemma 2.3 in \cite{ijac16}. 
 
\begin{lemma} \cite[Lemma 2.6]{DMS-coprimecomm}  \label{zero} Assume that $G=QH$ is a finite group with a normal  nilpotent subgroup $Q$ and a subgroup $H$ 
such that   $(|Q|,|H|)=1$ and  $|Q:C_Q(x)|\leq m$ for all $x\in H$. Then the order of $[Q,H]$ is $m$-bounded.
\end{lemma}

We will need a technical lemma, which provides a key result for the proofs of Theorem \ref{main1} and  Theorem \ref{main2}. 
\begin{lemma}\label{lem:keysol}
Let $G$ be a finite soluble group such that $\prp^* (G) \ge \epsilon >0$. 
 If  $p> (2/\epsilon)^{6/\epsilon}$, then every Sylow $p$-subgroup of $G$ is contained in $F(G)$.
\end{lemma} 
\begin{proof} 
Let $p$ be a prime such that $p> (2/\epsilon)^{6/\epsilon}$, and let $P$ be a Sylow $p$-subgroup of $G$. We need to show that $P\leq F(G)$. Suppose that this is false. Without loss of generality we may assume that $O_p(G)=1,$ so that $F(G)$ is a $p'$-subgroup. 

Write $F(G)=S_1\times\dots\times S_k,$ where $S_i$ is a Sylow $p_i$-supgroup of $F(G)$. Note that $S_i$ is normal in $G$ and it is contained in the Sylow $p_i$-subgroup  $Q_i$ of $G$ having the property that $\pr (P,  Q_i) \ge \epsilon.$ It  follows from Lemma \ref{lem:trivial} (2) that $\pr(P,S_i) \ge \epsilon$. 
In view of Corollary \ref{cor:easy} (1) we deduce that $[P, S_i]=1$.  Consequently, $P$ centralizes $F(G)$. A well-known property of finite soluble groups is that the centralizer of $F(G)$ is contained in $F(G)$. Thus, we have a contradiction, which proves the result.
\end{proof}

We will now prove Theorem \ref{main1} in the case of soluble groups.

\begin{proposition}\label{soluble}
Let $G$ be a finite soluble group  such that $\prp^* (G) \ge \epsilon >0$. 
 Then $F_2(G)$ has $\epsilon$-bounded index in $G$.
\end{proposition}
\begin{proof}
By Lemma \ref{lem:keysol}, if  $P$ is a Sylow $p$-subgroup of $G$ for a prime $p> (2/\epsilon)^{6/\epsilon},$  then $P$ is contained in $F(G)$. 
 Passing to the quotient over $F(G)$, we assume that the prime divisors of $|G|$ are smaller than $(2/\epsilon)^{6/\epsilon}$. In particular, the cardinality of $\pi(G)$ is  $\ep$-bounded. We will  show that under this assumption $F(G)$ has $\epsilon$-bounded index in $G$ (see also Remark \ref{oss} hereafter). This will imply the desired result. 

Let $p$ be a prime divisor of $|G|$ and let $P$ be a Sylow $p$-subgroup of $G$.

First assume that $O_p(G)=1$, in which case $F(G)$ is a $p'$-subgroup. 
Write
\[F(G)=S_1\times\dots\times S_k,\]
where $S_i$ is a Sylow $p_i$-subgroup of $F(G)$
 and $k$ is $\ep$-bounded. 
For each $p_i$, there exists a Sylow $p_i$-subgroup $Q_i$ of $G$ such that $\pr(P,Q_i) \ge \epsilon$. As $S_i$ is normal in $G$, it follows that $S_i\le Q_i$, therefore $\pr(P,S_i) \ge \epsilon$ by Lemma \ref{lem:trivial} (2). 
It follows from Lemma \ref{lem:H_0} (2) that there exists a subgroup $P_i$ of $P$ 
such that $|P:P_i| \le 2/\epsilon$ and $|S_i: C_{S_i}(x)| \le (2/\epsilon)^{6/\epsilon}$  for every $x \in P_i$. 
By Lemma \ref{zero},  
 the commutator subgroup
 $[S_i,P_i]$ has $\epsilon$-bounded order. 
Thus $C_i=C_{P_i}([S_i,P_i])$ has  $\epsilon$-bounded index in $P_i$. Moreover 
\[[S_i,C_i,C_i]\le[S_i,P_i,C_i]=1.\]
 As $C_i$ acts coprimely on $S_i$, it follows that $C_i$ centralizes $S_i$. Thus $|P_i:C_{P_i}(S_i)|$ is $\epsilon$-bounded, and therefore $|P:C_{P}(S_i)|$  is  $\epsilon$-bounded.
  Since  $k$ is $\ep$-bounded, we deduce that  $|P:C_{P}(F(G))|$  is  $\epsilon$-bounded. As $C_G(F(G))\le F(G)$ and $F(G)$ is a $p'$-group, it follows that $C_P(F(G))=1$. Thus, $P$ has $\epsilon$-bounded order.
  
Now we drop the assumption that $O_p(G)=1$ and apply the previous argument to $G/O_p(G)$. We deduce that the order of $P/O_p(G)$ is $\epsilon$-bounded. Since there are only $\epsilon$-boundedly many primes dividing $|G|$, it follows that the order of $G/F(G)$ is $\epsilon$-bounded. This concludes the proof.
\end{proof} 

\begin{remark}\label{oss}  The above proof shows that if  under the hypotheses of Proposition \ref{soluble} the order of $G$ is divisible by only $s$ primes, then $F(G)$ has $(\ep, s)$-bounded index in $G.$
\end{remark}

The following example shows that in Proposition \ref{soluble} the index of $F(G)$
can be arbitrarily large.

\begin{example}{\rm
Let $C_1, \dots, C_s$ be cyclic groups of different odd prime orders, and
let $H$ be the direct product of the groups $C_i.$ Let $A$ be the elementary
abelian group of order $2^s$ and choose a basis $a_1, \dots, a_s$ of $A.$ Let $G$ be 
the semidirect product of $H$ by $A,$ where for any $i, j$ the involution $a_i$
inverts the elements of
 $C_i$ and centralizes $C_j.$ Obviously, here the action of $A$ on $H$
is well-defined. It is easy to check that whenever $P$ and $Q$ are Sylow
subgroups of $G$ of coprime orders we have $\pr(P, Q) \ge 1/2.$ On the other
hand, $F(G)=H$ and $|G : H| = 2^s,$ which can be arbitrarily large.}\end{example}

Recall that the generalized Fitting subgroup $F^*(G)$ of a finite group
$G$ is the product of the Fitting subgroup $F(G)$ and all subnormal quasisimple subgroups; here a group is quasisimple if it is perfect and its
quotient by the centre is a nonabelian simple group. 

Now we are ready to prove Theorem \ref{main1}.

\begin{proof}[Proof of Theorem \ref{main1}]
Let $G$ be a finite group
such that $\prp^* (G) \ge \epsilon >0$. 
 We will show that the index of $F_2(G)$ in $G$ is $\ep$-bounded.

For the soluble radical $R=R(G)$ we have $\prp^* (R) \geq\epsilon$. Therefore,  by Proposition \ref{soluble}, the index of $F_2(R)$ in $R$ is $\ep$-bounded. Since $F_2(R) \le F_2(G)$, it is sufficient to prove that the order of $G/R$ is $\ep$-bounded. Note that  $\prp^*(G/R)\geq\epsilon$. We now assume that the soluble radical  $R$ of $G$ is trivial and our aim is to show that the order of $G$ is $\ep$-bounded. 
 
 As under our hypotheses the generalized Fitting subgroup $F^*(G)$ coincides with the socle of $G$, 
  we can apply Corollary \ref{cor:simple-product} to deduce that the order of  $F^*(G)$ is $\ep$-bounded. 
  Thus the index of $C_G(F^*(G))$ in $G$ is $\ep$-bounded, too. 
    Since $C_G(F^*(G)) \le F^*(G)$ (see, for instance, \cite[Theorem 9.8]{Isaacs}),  
		it follows that the order of $G$ is $\ep$-bounded. 
\end{proof}

We conclude with a proof of Theorem \ref{main2}.

\begin{proof}[Proof of Theorem \ref{main2}]
Let $G$ be a finite soluble group, and let  $\epsilon>0$ be a real number such that for any prime $p\in\pi(G)$ there is a Sylow $p$-subgroup $P$ and a Hall $p'$-subgroup $H$ of $G$ such that $\pr(P,H) \ge \epsilon$. We will show that the index of $F(G)$ in $G$ is $\ep$-bounded.

 We can assume that $G$ is not of prime power order. As any Hall $p'$-subgroup of $G$ contains a Sylow $q$-subgroup, it follows that $\prp^*(G) \geq\epsilon$.

Let $P$ be a Sylow $p$-subgroup of $G$. We wish to show that $P\cap F(G)$ has $\ep$-bounded index in $P$. Without loss of generality we may assume that  $O_p(G)=1$, in which case
 $F(G)$ is a $p'$-subgroup contained in any Hall $p'$-subgroup of $G$. It follows from Lemma \ref{lem:trivial} (2) that $\pr(P,F(G)) \ge \epsilon.$
By Lemma \ref{lem:H_0}  there exists a subgroup $P_0$ of $P$ 
such that $|P:P_0| \le 2/\epsilon$ and $|F(G): C_{F(G)}(x)| \le (2/\epsilon)^{6/\epsilon}$  for every $x \in P_0$. 
By Lemma \ref{zero},   the commutator subgroup
 $[F(G),P_0]$ has $\epsilon$-bounded order. 
Thus $C=C_{P_0}([F(G),P_0])$ has  $\epsilon$-bounded index in $P$. Moreover 
\[[F(G),C,C]\le[F(G),P_0,C]=1.\]
 As $C$ acts coprimely on $F(G)$, it follows that $C$ centralizes $F(G)$. Thus $|P:C_{P}(F(G))|$ is $\epsilon$-bounded. As $C_G(F(G))\le F(G)$ and $F(G)$ is a $p'$-group, it follows that $C_P(F(G))=1$ and so $P$ has $\epsilon$-bounded order. Thus, we have shown that for any $p\in\pi(G)$ and any Sylow $p$-subgroup $P$ the index of $P\cap F(G)$ in $P$ is $\ep$-bounded, say at most $m.$ 
In particular, $P \le F(G)$ for all primes $p>m$. 
  Since $m$ does not depend on $p$ and is $\ep$-bounded, it follows that the order of $G/F(G)$ is $\ep$-bounded, as required. 
\end{proof}


\begin{thebibliography}{10}
\bibitem{asch}  M. Aschbacher, Finite group theory. Second edition. Cambridge Studies in Advanced Mathematics, 10. Cambridge University Press, Cambridge, 2000. 



\bibitem{bender}  H. Bender, On Groups with Abelian Sylow $2$-Subgroups, Math. Z.  {\bf 117}, 164--176 (1970).
\bibitem {bgmn} T. Burness, R.\,M. Guralnick, A. Moret\'o, G. Navarro, On the commuting probability of $p$-elements in finite groups, Algebra and Number Theory, {\bf 17} (2023), No. 6, 1209--1229.
\bibitem{DS}  E. Detomi, P. Shumyatsky,   On the commuting probability for
subgroups of a finite group, 	Proc. Roy. Soc. Edinburgh Sect. A  {\bf 152} (2022), 1551--1564.
\bibitem{DMS-coprimecomm} E. Detomi, M. Morigi, P. Shumyatsky,  Strong conciseness of coprime and anti-coprime commutators. Ann. Mat. Pura Appl. {\bf 200},  945--952 (2021)
\bibitem{DS-commprob} E. Detomi, P. Shumyatsky, On the commuting probability for subgroups of a
finite group, Proc. Roy. Soc. Edinburgh Sect. A 152 (2022), 1551--1564.
\bibitem{eberhard} S. Eberhard, Commuting probabilities of finite groups, 
Bull. London Math. Soc. {\bf 47} (2015), 796--808.

\bibitem{Erf} A. Erfanian, R. Rezaei, P. Lescot, On the Relative Commutativity Degree of
a Subgroup of a Finite Group, Comm. Algebra {\bf 35} (2007), 4183--4197.
\bibitem{fetho} W. Feit, J. G. Thompson,  Solvability of groups of odd order, Pacific Journal of Mathematics, {\bf 13} (1963), 775--1029.

  \bibitem{GAP4}
  The GAP~Group, \emph{GAP -- Groups, Algorithms, and Programming, 
  Version 4.12.2}; 
  2022, {https://www.gap-system.org}.

\bibitem{gg} R. Gilman, D.  Gorenstein,  Finite groups with Sylow 2-subgroups of class two. I, II. Trans. Amer. Math. Soc. {\bf 207} (1975), 1--101; ibid. {\bf 207} (1975), 103--126. 


 \bibitem{GLS} D. Gorenstein, R. Lyons, R. Solomon, The classification of the finite simple groups. Number 3. Part I. Chapter A. Almost simple $K$-groups. Mathematical Surveys and Monographs, 40.3. American Mathematical Society, Providence, RI, 1998.

\bibitem{gr} R.\,M. Guralnick, G.\,R. Robinson, On the commuting probability in finite groups, J. Algebra {\bf 300} (2006), 509--528.


\bibitem{huppert} B. Huppert,
Endliche Gruppen. I. 
Die Grundlehren der mathematischen Wissenschaften, Band 134 Springer-Verlag, Berlin-New York 1967.

\bibitem{characters} I.\,M. Isaacs, Character theory of finite groups. Pure and Applied Mathematics, No. 69. Academic Press, New York-London, 1976.

\bibitem{Isaacs}  I.\,M. Isaacs, Finite Group Theory, Graduate Studies in Mathematics, 92.
American Mathematical Society, Providence, RI, 2008.

\bibitem{kaz} L.\,S. Kazarin, Burnside’s $p^\alpha$-lemma, Mat. Zametki 48 (1990),  45--48, 158; translation in
Math. Notes  {\bf 48} (1990), 749--751.

\bibitem{ijac16}
 E.\,I. Khukhro, P.  Shumyatsky,  Almost Engel finite and profinite groups,  { Internat. J. Algebra Comput.} 
{\bf 26}   (2016),  973--983.
 

\bibitem{twop} C.\,H. Li,  X. Li,
On permutation groups of degree a product of two prime-powers,
Comm. Algebra {\bf 42} (2014), 4722--4743.


\bibitem{nath} R.\,K. Nath, M.\,K. Yadav, Some results on relative commutativity degree, Rend. Circ. Mat. Palermo {\bf 64} (2) (2015), 229--239. 
\bibitem{neumann} P. M. Neumann, Two Combinatorial Problems in Group Theory, Bull.
 London Math. Soc. {\bf 21}  (1989), 456--458.
 

\bibitem{walter} J.\,H. Walter, 
The characterization of finite groups with abelian Sylow 2-subgroups.
Ann. of Math. {\bf 89} (1969), 405--514.


\bibitem{z} {{K. Zsigmondy}}, Zur Theorie der Potenzreste,
{Monatsh. f\"ur Math. u. Phys.} {\bf 3} (1892), 265--284.

\end{thebibliography}
\end{document}